\documentclass[12pt]{article} 

\usepackage[utf8]{inputenc} 

\usepackage{geometry} 
\usepackage{graphicx} 

\usepackage{booktabs} 
\usepackage{array} 
\usepackage{paralist} 
\usepackage{verbatim} 
\usepackage{subfig} 

\usepackage{fancyhdr} 
\usepackage{sectsty}
\allsectionsfont{\sffamily\mdseries\upshape} 

\usepackage[nottoc,notlof,notlot]{tocbibind} 
\usepackage[titles,subfigure]{tocloft} 

\usepackage{url} 
\usepackage{amsthm, amssymb,amsmath} 
\newtheorem{theorem}{Theorem}

\title{Nonexistence of $srg(19,6,1,2)$: Combinatorial Proof}
\author{Reimbay Reimbayev}
\date{} 

\begin{document}
\maketitle

\begin{abstract}
An $srg(19,6,1,2)$ is the graph with the smallest parameter set in the family of strongly regular graphs with parameters $\lambda=1$ and $\mu=2$ for which the respective graph doesn't exist. The proof of that fact is based on algebraic arguments, particularly, on the Integrality Test, the very usefull tool for studying strongly regular graphs. To our best knowledge, there have not been proofs of pure combinatorial nature. In this short paper, we have decided to fill in this gap.
\end{abstract}


\bigskip
In the family of strongly regular graphs with parameters $\lambda = 1$ and $\mu = 2$, the graph $srg(19,6,1,2)$ is not a `superstar' like an $srg(99,14,1,2)$, which is under the search for at least half a century to the moment \cite{Conway}. Moreover, it doesn't even exist - a well established fact. But this is the first graph that doesn't exist in the line of possible ones. For the first two possible values for the parameters $n$ and $k$, respectively the order of the graph and valency of its vertices, the graphs exist: $K_3$ (for $k=2$) - a complete graph on three vertices , $P_9$ (for $k=4$) - Paley graph on nine vertices.

For the parameter $k=6$, which is the next possible, i.e. our case, the graph doesn't exist due to the Integrality Test, an incredibly useful tool in studying strongly regular graphs. The tool that uses multiplicities of eigenvalues of an adjacency matrix of a graph to determine its nonexistence. But Integrality Test is only a necessary condition for existance of a strongly regular graph and thus there are still some cases that slips through it, namely for the parameters of $k=14, 22, 112, 994.$ 

Out of these four graphs that pass the Integrality Test, only for $k=22$ the graph has been shown to exist \cite{Berlekamp}. For the other ones, the question of existence has not been resolved. For that reason, we think, any attempt to prove the existance of such graphs without referral to Integrality Test might be useful.

\bigskip

By definition, the graph is strongly regular if a pair of its vertices has exactly $\lambda$ common neighbors given they are adjacent, or $\mu$ common neighbors otherwise \cite{Gordon, Brouwer}. Another way of defining strongly regular graphs, perhaps more precise as it cuts away some trivial cases like complete graphs, is by using spectral graph theory, by which the finite graph is strongly regular if its spectrum consists of exactly three eigenvalues, one of which is $k$ with multiplicity one \cite{BrouwerMaldeghem}.  

For simplicity, let us call an $srg(19,6,1,2)$ further on just $G$. Below is the formal statement we are going to prove.

\begin{theorem}
srg(19,6,1,2) doesn't exist.
\end{theorem}

\begin{proof}
Choose any triangle, $K_3$, with vertices $a,b,c$ from $G=srg(19,6,1,2)$. We will denote $A = \{v\in G | va \in E(G), v \neq b,c \}, B = \{v\in G | vb \in E(G), v \neq a,c \}, C = \{v\in G | vc \in E(G), v \neq a,b \}$, the sets of vertices adjacent to a given vertex of the original triangle except the vertices of the triangle.

One more set is needed $W = V(G)\setminus (A\cup B\cup C \cup \{ a,b,c\})$, the set of vertices at distance 2 from triangle $\{ a,b,c\}$. Notice that $\{ a,b,c\}$, $A, B, C$, and $W$ - is a partition of $V(G)$, and $|A| = |B| = |C| = |W| = 4$.
Now, here is the goal we want to achieve: to show that an induced subgraph on vertices $A \cup B \cup C, (G[A\cup B\cup C])$ cannot satisfy both $\lambda=1$ (Condition I) and $\mu=2$ (Condition II). First, we show that it doesn't contain triangles. Vertex $w$ from $W$ is adjacent to exactly $2$ vertices from $A$, as $wa \notin E(G)$. The same hold true regarding $B$ and $C$ (Figure 1). But $deg(w) = 6$, and an induced subgraph on W is an empty graph. That means there are exactly $12 (3 \cdot 4)$ triangles with one vertex on $W$ and two on $A \cup B \cup C$. In total there 19 triangles in G ($\frac{nk}{6}$, Proposition 3 \cite{ReiLowerBound}), out of which 7 have at least one vertex from \{ a,b,c\}. Thus, $19-7-12 = 0$, and $A\cup B\cup C$ has no triangles.

\begin{figure}
	\includegraphics[width=0.4\textwidth]{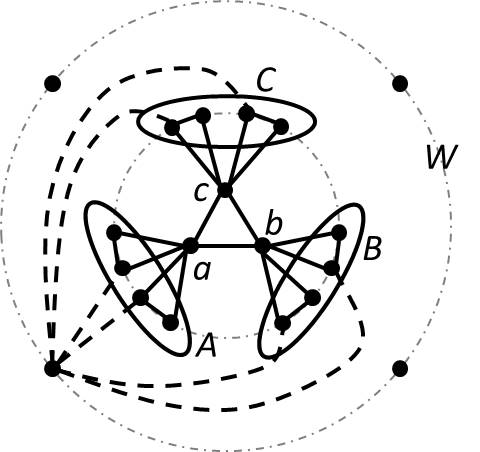}
		\centering
		\caption{The structure of $srg(19,6,1,2)$}
		\label{fig3}
\end{figure}

For any $\tilde{a} \in A$ there exists unique $\tilde{b} \in B$ such that $\tilde{ab} \in E(G)$. This is due to Condition II because $\tilde{a}b \notin E(G)$. The function given by this adjacency relationship, say $f: A \rightarrow B$ is bijective. Suppose not. Then there exist $\tilde{b}_1$ and $\tilde{b}_2$ both adjacent to $\tilde{a} \in A$. But then $\tilde{a}$ and $b$, two non-adjacent vertices, have three common neighbors, - $a$, $\tilde{b}_1, \tilde{b}_2$, which is impossible. Similarly, there cannot be $\tilde{a}_1, \tilde{a}_2 \in A$ and $\tilde{b} \in B$, such that $\tilde{a}_1 \tilde{b} \in E(G)$ and $\tilde{a}_2 \tilde{b} \in E(G)$. We can continue in this way and set a bijection $B \rightarrow C$, then another one $C \rightarrow A$. Bijectivity of the relationship $A\rightarrow B\rightarrow C\rightarrow A$, or simply permutation of vertices of $A$, makes the induced subgraph on $A\cup B\cup C$ regular of degree three. Not two because we have to take into account pairwise connectedness of vertices in $A$, as well as in $B$ and $C$. Thus, without these edges, $G[A\cup B\cup C]$ consists of cycles of length 3, 6, 9, or 12. The cycles of lenght 3, triangles, are not possible as we already agreed before. By the same reason $C_9$ is also impossible, otherwise $G[A\cup B\cup C] = C_9 + C_3$ plus additional inner chords, which gives us a union of two cycles (and additional chords or bridges), one of which is a triangle. There has left two cases, which we consider separately.

\textbf{Case 1}:$G[A\cup B\cup C] = C_6 + C_6$ plus chords and/or connectors.
Consider a 6-cycle, with possible chords that we will ignore altogether, $a_1 b_1 c_1 a_2 b_2 c_2$, where $a_i \in A, b_i \in B, c_i \in C, i=1,2$. Each edge is a base of a triangle with the third vertex on $W = \{w_1,w_2,w_3,w_4\}$. The goal is to show that the proper distribution of the vertices of the given six triangles among four elements of $W$ is impossible. Denote $w_1$ the vertex of the first triangle based on the edge $a_1 b_1$ and move anti-clockwise as depicted on the Figure 2 (case 1). We cannot chose $w_1$ for the next triangle with basis on $b_1 c_1$ or else the edge $b_1 w_1$ belongs to two trianlges. Chose $w_2$. Next triangle based on $c_1 a_2$ necesseraly should have a vertex distinct from both previous ones: $w_1$ and $w_2$. Otherwise $b_1 c_1$ belongs to two triangles $b_1 c_1 w_1$ and $b_1 c_1 w_1$. So $w_3\neq w_1, w_2$.

\begin{figure}
	\includegraphics[width=0.8\textwidth]{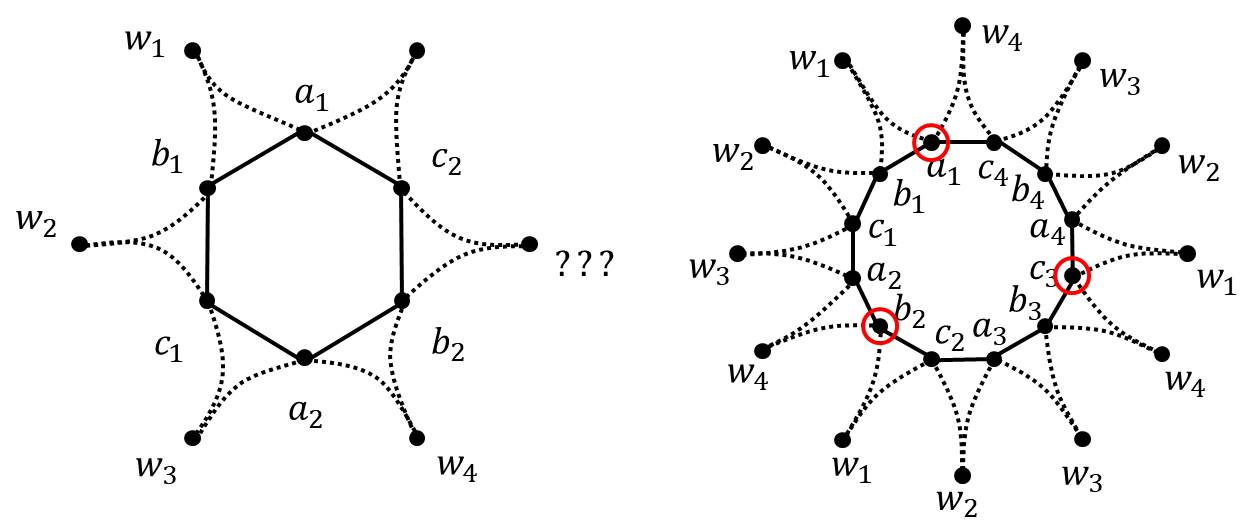}
		\centering
		\caption{Illustration for the Theorem 1: left - for case 1, right - for case 2.}
		\label{fig4}
\end{figure}

Last step. Consider triangle $a_2 b_2 w_4$, where $w_4\neq w_2, w_3$ due to the same reasoning as before. But $w_4\neq w_1$ as well! Assume not, and $w_4=w_1$. We know that every vertex belongs to exactly three triangles ($k=6$). For $w_1$ those are $w_1 a_1 b_1$, $w_1 a_2 b_2$, and $w_1 x y$, where $x$ and $y$ are both belong to $C$. But that is not possible, otherwise the edge $xy$ belongs to $w_1 x y$ and $c x y$, two distinct triangles at the same time. We run out of vertices from $W$ to assign for the next triangle in the line.

\textbf{Case 2}:$G[A\cup B\cup C] = C_{12}$ plus chords.
The only possible way of distributing the vertices of $W$ is shown in the figure. Again, this is due to the same reasons as in Case 1. We have also established earlier that $G[W] = \overline{K_4}$ is an empty graph. Thus $w_1 w_4 \notin E(G)$. But this two vertices now have three common neighbors, namely $a_1, b_2, c_3$, which is not possible. That completes the proof.

\end{proof}

Similarly, it can be proved, with a bit more work, for the next parameter of $k$. Instead, using the Integrality Test \cite{West, Cvetcovic2} eliminates decisively all the graphs up to $k=14$, which is also called Conway-99 graph and is among five problems Conway has offered to solve \cite{Conway}. Three more values of $k$, $k=22, 112, 994$, also pass the Integrality Test, out of which one for $k=22$, to ones amusement, has been already constructed \cite{Berlekamp}.

\bigskip
In conclusion, in this short paper we have proved nonexistance of the lowest possible graph from the family of strongly regular graphs with parameters $\lambda = 1$ and $\mu =2$ using strictly combinatorial arguments. Although, the more important graph from the same family is an $srg(99,14,1,2)$, existance of which is still undefined, we do not think that our work is a futile exercise but rather an attempt to look from another perspective on the problem. The search for an $srg(99,14,1,2)$ is still an on-going process.


\end{document}